\documentclass[11pt]{article}
\usepackage{t1enc}
\usepackage[latin1]{inputenc}
\usepackage[english]{babel}
\usepackage{amsmath,amsthm}
\usepackage{amsfonts}
\usepackage{latexsym}
\usepackage[dvips]{graphicx}
\usepackage{graphicx}
\usepackage{mathrsfs}
\usepackage{color}
\newcounter{countclaim}

\date{} \textwidth 16cm \textheight 22cm \topmargin 0 cm \hoffset
-1.5 cm \voffset 0cm

\newtheorem{theorem}{Theorem}[section]
\newtheorem{lemma}[theorem]{Lemma}

\newtheorem{definition}[theorem]{Definition}
\newtheorem{corollary}[theorem]{Corollary}

\title{The rank of a complex unit gain graph\\ in terms of the matching number}
\author{Shengjie  He$^{\rm a}$, Rong-Xia Hao$^{\rm a}$\footnote{Corresponding author.
Emails: he1046436120@126.com (Shengjie  He), rxhao@bjtu.edu.cn (Rong-Xia Hao), fengming.dong@nie.edu.sg  (Fengming Dong)}, Fengming Dong$^{\rm b}$\\
{\small\em $^{\rm a}$Department of Mathematics, Beijing Jiaotong University, Beijing,
100044, China}\\
{\small\em $^{\rm b}$National Institute of Education,
  Nanyang Technological University, Singapore}\\
}

\begin{document}
\baselineskip 0.55cm \maketitle

\begin{abstract}

A complex unit gain graph (or ${\mathbb T}$-gain graph) is a triple
$\Phi=(G, {\mathbb T}, \varphi)$ (or $(G, \varphi)$ for short) consisting
of a simple graph $G$, as the underlying graph of $(G, \varphi)$, the set
of unit complex numbers $\mathbb{T}= \{ z \in C:|z|=1 \}$
and a gain function $\varphi: \overrightarrow{E} \rightarrow \mathbb{T}$ with
the property that $\varphi(e_{i,j})=\varphi(e_{j,i})^{-1}$.
In this paper, we prove that
$2m(G)-2c(G) \leq r(G, \varphi) \leq 2m(G)+c(G)$,
where $r(G, \varphi)$, $m(G)$ and $c(G)$ are the rank of the Hermitian
adjacency matrix $H(G, \varphi)$, the matching number and the cyclomatic number
of $G$, respectively. Furthermore, the complex unit gain graphs
$(G, \mathbb{T}, \varphi)$ with $r(G, \varphi)=2m(G)-2c(G)$ and $r(G, \varphi)=2m(G)+c(G)$
are characterized. These results generalize the corresponding known results about
undirected graphs, mixed graphs and signed graphs. Moreover,
we show that $2m(G-V_{0}) \leq r(G, \varphi) \leq 2m(G)+b(G)$
holds for any subset $V_0$ of $V(G)$ such that $G-V_0$ is acyclic and
$b(G)$ is the minimum integer $|S|$ such that $G-S$ is bipartite
for $S \subset V(G)$.

\end{abstract}

{\bf Keywords}: Complex unit gain graph; rank; matching number; cyclomatic number.

{\bf MSC}: 05C50

\section{Introduction}

Throughout this paper, we only consider the simple graphs, i.e., without multiedges and loops.
Let $G=(V,E)$ be a simple graph with
$V=\{v_{1}, v_{2}, \cdots, v_{n} \}$.
Whenever $v_iv_j\in E$,
let $e_{i,j}$ denote the ordered pair $(v_i,v_j)$.
Thus $e_{i,j}$ and $e_{j,i}$ are considered to be distinct.
Let $\overrightarrow{E}$ denote the set of
$\{e_{i,j},e_{j,i}: v_iv_j\in E\}$.
Clearly $|\overrightarrow{E}|=2|E|$.
Let $\mathbb{T}$ be the set of all
complex numbers $z$ with $|z|=1$
and let $\varphi$ be an arbitrary  mapping $\varphi:
\overrightarrow{E} \rightarrow \mathbb{T}$ such that  $\varphi(e_{i,j})=\varphi(e_{j,i})^{-1}$
whenever $v_iv_j\in E$.
We call $\Phi=(G, \mathbb{T}, \varphi)$
a complex unit gain graph.
For convenience, we write $(G, \varphi)$ for
a complex unit gain graph $\Phi=(G, \mathbb{T}, \varphi)$ in this paper. We refer to \cite{BONDY} for undefined terminologies and notations.

The {\it adjacency matrix} $A(G)$ of $G$ is the $n \times n$ matrix $(c_{i,j})$,
where $c_{i,j}=1$ whenever $v_iv_j\in E$,
and $c_{i,j}=0$  otherwise.
The adjacency matrix associated to the complex unit gain graph $(G, \varphi)$ is the $n \times n$ complex matrix $H(G, \varphi)=(a_{i,j})$, where
$a_{i,j}=\varphi(e_{i,j})$ whenever $v_iv_j\in E$,
and $a_{i,j}=0$  otherwise.

Observe that $H(G, \varphi)$ is Hermitian
and its eigenvalues are
real \cite{WLG}.
The rank of the complex unit gain graph $(G, \varphi)$
is defined to be the rank of the matrix $H(G, \varphi)$,
denoted by $r(G,\varphi)$.
Thus, 
$r(G,\varphi)=p^{+}(G, \varphi)+n^{-}(G, \varphi)$,
where $p^{+}(G, \varphi)$ (resp. $n^{-}(G, \varphi)$),
called the
{\it positive inertia index}
(resp. the {\it negative inertia index})
of $(G, \varphi)$,
is the number of positive eigenvalues
(resp. negative eigenvalues) of
$H(G, \varphi)$.

The value $c(G)=|E(G)|-|V(G)|+\omega(G)$, is called the {\it cyclomatic number} of a graph $G$,
where $\omega(G)$ is the number of connected components of $G$.
A set of pairwise independent edges of $G$ is called a $matching$, while a matching with the maximum
cardinality is a $maximum$ $matching$ of $G$. The $matching$ $number$ of $G$, denoted by $m(G)$,
is the cardinality of a maximum matching of $G$.
For a complex unit gain graph $(G, \varphi)$, the matching number and cyclomatic number
of $(G, \varphi)$  are defined to be the matching number and cyclomatic number of its underlying graph, respectively. Denote by $P_n$ and $C_n$ a path and cycle on $n$ vertices,
respectively.

In recent years, the study on the complex unit gain graphs has received more and more attentions.
Lu et al. \cite{WLG} studied the relation between the rank of a complex unit gain graph and the rank of its underlying graph.
In \cite{YGH}, the positive inertia and negative inertia of a complex unit gain cycle were characterized by Yu et al.
In \cite{FYZUNIT}, Wang et al.
investigated the determinant of the Laplacian matrix of a complex unit gain graph.
In \cite{REFF},
Reff generalized some fundamental concepts from spectral graph theory to complex unit gain graphs and defined the adjacency, incidence and Laplacian matrices of them.

The study on the relation between the rank of graphs and other topological structure parameters has been a popular subject in the graph theory.
Mohar \cite{MOHAR} and Li \cite{LXL} introduced the Hermitian adjacency matrix of a mixed graph and presented some basic properties of the rank of the mixed graphs independently.
In \cite{WANGLONG}, Wang et al. characterized the relation among the rank and the matching number and the independence number of an undirected graph.
 In \cite{WDY},
Ma et al. investigated the relation between the skew-rank of an oriented graph and the matching number of its underlying graph.
The relation between the rank of a mixed graph and the
matching number was discussed by Li et al. \cite{LSC}.
Huang et al. \cite{HLSC} researched the relation between the skew-rank of an oriented graph and its independence number.
For other research of the rank of a graph, one may be referred to those in \cite{BEVI,GUT,MAH2,MOHAR2,RULA,WONGEUR}.

It is obvious that an undirected graph $G$ is just a complex unit gain graph $\Phi=(G,{\mathbb T},\varphi)$ with $\varphi(\overrightarrow{E})\subseteq  \{1\}$.
In \cite{WANGLONG}, Wang et al. researched the the relation among the rank, the matching number and the cyclomatic number of an undirected graph.

\begin{theorem}[\cite{WANGLONG}]\label{un1}
Let $G$ be a simple connected undirected graph. Then
$$
2m(G)-2c(G) \leq r(G) \leq 2m(G)+c(G).
$$
\end{theorem}

Let $T_G$ be the graph obtained from $G$
by contracting each cycle of $G$ into a vertex (called a {\it cyclic vertex}).
The the lower bound and the upper bound for $r(G)$ of the undirected graphs
$G$ in Theorem~\ref{un1} are obtained in \cite{SONG} and \cite{WANGL}, respectively.

\begin{theorem}[\cite{SONG}]\label{un2}
Let $G$ be a simple connected undirected graph.
Then $r(G)=2m(G)-2c(G)$ if and only if  all the following conditions hold for $G$:

{\em(i)} the cycles (if any) of $G$ are pairwise vertex-disjoint;

{\em(ii)} the length of each cycle (if any) of $G$ is a multiple of 4;

{\em(iii)} $m(T_{G})=m(G-O(G))$, where $O(G)$ is the set of vertices in cycles of $G$.
\end{theorem}

\begin{theorem}[\cite{WANGL}]\label{un3}
Let $G$ be a simple connected undirected graph.
Then $r(G)=2m(G)+c(G)$ if and only if  all the following conditions hold for $G$:

{\em(i)} the cycles (if any) of $G$ are pairwise vertex-disjoint;

{\em(ii)} each cycle (if any) of $G$ is odd;

{\em(iii)} $m(T_{G})=m(G-O(G))$, where $O(G)$ is the set of vertices in cycles of $G$.
\end{theorem}

A $mixed$ $graph$ $\widetilde{G}$ is a graph where both directed and undirected edges may exist.
The Hermitian-adjacency matrix of a mixed graph $\widetilde{G}$ of order $n$ is the $n \times n$ matrix $H(\widetilde{G}) = (h_{kl})$, where $h_{kl}=-h_{lk}=\mathbf{i}$ if
$v_{k} \rightarrow v_{l}$, where $\mathbf{i}$ is the imaginary number unit and $h_{kl}=h_{lk}=1$ if $v_{k}$ is connected to $v_{l}$ by an undirected edge, and $h_{kl}=0$ otherwise.
For a mixed cycle $\widetilde{C}$, the $signature$ of $\widetilde{C}$, denoted by $\eta(\widetilde{C})$, is defined as $|f-b|$, where $f$ denotes the number of forward-oriented edges and $b$ denotes the number of backward-oriented edges of $\widetilde{C}$. It is obvious that a mixed graph $\widetilde{G}$ is just a complex unit gain graph $\Phi=(G,{\mathbb T},\varphi)$ with $\varphi(\overrightarrow{E})\subseteq  \{1, \boldsymbol{i}, -\boldsymbol{i} \}$.

For mixed graphs, the relation among the rank, the matching number and the cyclomatic number was investigated independently in \cite{LSC} and \cite{TFL}, respectively.

\begin{theorem}[\cite{LSC,TFL}]\label{mixed1}
Let $\widetilde{G}$ be a connected mixed graph. Then
$$
2m(G)-2c(G) \leq r(\widetilde{G}) \leq 2m(G)+c(G).
$$
\end{theorem}

\begin{theorem}[\cite{LSC,TFL}]\label{mixed2}
Let $\widetilde{G}$ be a connected mixed graph.
Then $r(\widetilde{G})=2m(G)-2c(G)$ if and only if  all the following conditions hold for $\widetilde{G}$:

{\em(i)} the cycles (if any) of $\widetilde{G}$ are pairwise vertex-disjoint;

{\em(ii)} each cycle $\widetilde{C_{l}}$ of $\widetilde{G}$ is even with $\eta(\widetilde{C_{l}}) \equiv   l \, (\rm{mod} \ 4)$;

{\em(iii)} $m(T_{G})=m(G-O(G))$, where $O(G)$ is the set of vertices in cycles of $\widetilde{G}$.
\end{theorem}

\begin{theorem}[\cite{LSC,TFL}]\label{mixed3}
Let $\widetilde{G}$ be a connected mixed graph.
Then $r(\widetilde{G})=2m(G)+c(G)$ if and only if all the following conditions hold for $\widetilde{G}$:

{\em(i)} the cycles (if any) of $\widetilde{G}$ are pairwise vertex-disjoint;

{\em(ii)} each cycle of $\widetilde{G}$ is odd with even signature;

{\em(iii)} $m(T_{G})=m(G-O(G))$, where $O(G)$ is the set of vertices in cycles of $\widetilde{G}$.
\end{theorem}

A $signed$ $graph$ $(G, \sigma)$ consists of a simple graph $G=(V, E)$, referred to as its underlying graph,
and a mapping $\sigma: E  \rightarrow  \{ +, - \}$, its edge labelling. Let $C$ be a cycle of $(G, \sigma)$. The $sign$ of $C$ is defined by $\sigma(C)=\prod_{e \in C}\sigma(e)$.
A cycle $C$ is said to be positive or negative if $\sigma(C)=+$ or $\sigma(C)=-$, respectively.
It is obvious that a signed graph $(G, \sigma)$ is just a complex unit gain graph $\Phi=(G,{\mathbb T},\varphi)$ with $\varphi(\overrightarrow{E})\subseteq  \{1,-1\}$.

In~\cite{HSJ}, He et al.
characterized the relation among the rank, the matching number and the cyclomatic number of a signed graph.


\begin{theorem}[\cite{HSJ}]\label{signed1}
Let $(G, \sigma)$ be a connected signed graph. Then
$$
2m(G)-2c(G) \leq r(G, \sigma) \leq 2m(G)+c(G).
$$
\end{theorem}

\begin{theorem}[\cite{HSJ}]\label{signed2}
Let $(G, \sigma)$ be a connected signed graph.
Then $r(G, \sigma)=2m(G)-2c(G)$ if and only if  all the following conditions hold for $(G, \sigma)$:

{\em(i)} the cycles (if any) of $(G, \sigma)$ are pairwise vertex-disjoint;

{\em(ii)} for each cycle (if any) $C_{q}$ of $(G, \sigma)$, either $ q \equiv 0 \ (\rm{mod} \ 4)$ and $\sigma(C_{q})=+$ or $ q \equiv 2 \ (\rm{mod} \ 4)$ and $\sigma(C_{q})=-$;

{\em(iii)} $m(T_{G})=m(G-O(G))$, where $O(G)$ is the set of vertices in cycles of $(G, \sigma)$.
\end{theorem}

\begin{theorem}[\cite{HSJ}]\label{signed3}
Let $(G, \sigma)$ be a connected signed graph.
Then $r(G, \sigma)=2m(G)+c(G)$ if and only if  all the following conditions hold for $(G, \sigma)$:

{\em(i)} the cycles (if any) of $(G, \sigma)$ are pairwise vertex-disjoint;

{\em(ii)} each cycle (if any) of $(G, \sigma)$ is odd;

{\em(iii)} $m(T_{G})=m(G-O(G))$, where $O(G)$ is the set of vertices in cycles of $(G, \sigma)$.
\end{theorem}

In this paper,
we will find a lower bound and an upper bound for
$r(G,\varphi)$ in terms of $c(G)$ and $m(G)$,
where $c(G)$ and $m(G)$ are the
cyclomatic number and the matching number of $G$
respectively.
Moreover, the properties of the extremal graphs which attain the lower and upper bounds are investigated.
Our results generalize the corresponding results about undirected graphs, mixed graphs and signed graphs, which were obtained in \cite{WANGLONG}, \cite{LSC} and \cite{HSJ}, respectively.
Our main results are the following Theorems \ref{T30}, \ref{T50}, \ref{T40} and \ref{T60}.

The following Theorem \ref{T30} generalizes Theorems \ref{un1}, \ref{mixed1} and \ref{signed1}.

\begin{theorem}\label{T30}
For any connected complex unit gain graph $\Phi=(G,{\mathbb T},\varphi)$, we have
$$
2m(G)-2c(G) \leq r(G, \varphi) \leq 2m(G)+c(G).
$$
\end{theorem}

Let $(C_{n}, \varphi)$ ($n \geq 3$) be a complex unit gain cycle for $C_{n}=v_{1}v_{2} \cdots v_{n}v_{1}$ and
$\varphi(C_{n}, \varphi)=\varphi(v_{1}v_{2} \cdots v_{n}v_{1} )=\varphi(v_{1}v_{2})\varphi(v_{2}v_{3})  \cdots \varphi(v_{n-1}v_{n})\varphi(v_{n}v_{1})$,
where $Re(x)$ be the real part of a complex number $x$.

The following Theorem \ref{T50} generalizes Theorems \ref{un2}, \ref{mixed2} and \ref{signed3}.

\begin{theorem}\label{T50}
For any connected complex unit gain graph $\Phi=(G,{\mathbb T},\varphi)$,
$r(G, \varphi)= 2m(G)-2c(G)$
if and only if all the following conditions hold:
\begin{enumerate}
\item[(i)] the cycles (if any) of $(G, \varphi)$ are pairwise vertex-disjoint;

\item[(ii)] for each cycle (if any) $(C_{l}, \varphi)$ of $(G, \varphi)$, $\varphi(C_{l}, \varphi)=(-1)^{\frac{l}{2}}$ and $l$ is even;

\item[(iii)] $m(T_{G})=m(G-O(G))$, where $O(G)$ is the set of vertices in cycles of $(G, \varphi)$.
\end{enumerate}
\end{theorem}

The following Theorem \ref{T40} generalizes Theorems \ref{un3}, \ref{mixed3} and \ref{signed3}.

\begin{theorem}\label{T40}
For any connected complex unit gain graph $\Phi=(G,{\mathbb T},\varphi)$,
$r(G, \varphi)= 2m(G)+c(G)$
if and only if all the following conditions hold:
\begin{enumerate}
\item[(i)] the cycles (if any) of $(G, \varphi)$ are pairwise vertex-disjoint;

\item[(ii)] for each cycle (if any) $(C_{l}, \varphi)$ of $(G, \varphi)$, $Re(\varphi(C_{l}, \varphi)) \neq 0$ and $l$ is odd;

\item[(iii)] $m(T_{G})=m(G-O(G))$,
where $O(G)$ is the set of vertices in cycles of $(G, \varphi)$.
\end{enumerate}
\end{theorem}

Another upper bound and lower bound of the rank of a complex unit gain graph in terms of other parameters are obtained.

\begin{theorem}\label{T60}
For any connected complex unit gain graph
$\Phi=(G,{\mathbb T},\varphi)$, we have
$$
2\max_{V_0} m(G-V_{0}) \leq r(G, \varphi) \leq 2m(G)+b(G),
$$
where the maximum value of $m(G-V_{0})$
is taken over all proper subsets $V_0$ of $V(G)$
such that $G-V_0$ is acyclic and
$b(G)$ is the minimum integer $|S|$ such that $G-S$ is bipartite
for $S \subset V(G)$.
\end{theorem}

The rest of this paper is organized as follows. In Section 2, some useful lemmas are listed which will be used in the proof of our main results.
The proof of the Theorem \ref{T30} is presented in Section 3. In Section 4, the proof for Theorem \ref{T50} is given and the properties of the extremal complex unit gain graphs which attain the lower bound of Theorem \ref{T30} are discussed. In Section 5, the properties of the extremal complex unit gain graphs which attain the upper bound of Theorem \ref{T30} are researched, and the proof for Theorem \ref{T40} is shown. The proof of Theorem \ref{T60} and the discuss for further study are given in Section 6.

\section{Preliminaries}

We need the following known results and useful lemmas to prove our main result, which will be used in next sections.

For any $v_i\in V$, let $d_G(v_i)$ (or simply $d(v_i)$)
denote the {\it degree} of $v_i$ in $G$.
A vertex $v_i$ in $G$ is called a {\it pendant vertex}
if $d(v_i)=1$,
and is called a {\it quasi-pendant vertex}
if $d(v_i)\ge 2$ and $v_i$ is adjacent to
some pendant vertex.

For any $S\subseteq V$,
let $G[S]$ denote the subgraph of $G$ induced by $S$,
and let $G-S$ denote $G[V-S]$ when $S\ne V$.
For any $x\in V$, let $G-x$ simply denote $G-\{x\}$.
For an induced subgraph $H$ and $u\in V-V(H)$,
let $H+u$ denote $G[V(H)\cup \{u\}]$.
A spanning subgraph $G_0$ of $G$ is called
{\it elementary} if each component of $G_0$
is either $K_{2}$ or a cycle.

\begin{lemma} \label{L16}{\rm\cite{YGH}}
Let $(G, \varphi)$ be a complex unit gain graph.
\begin{enumerate}
\renewcommand{\theenumi}{\rm (\roman{enumi})}
\item If $(H,\varphi)$ is an induced subgraph of $(G, \varphi)$,
then $r(H, \varphi) \leq r(G, \varphi)$.

\item If $(G_{1}, \varphi), (G_{2}, \varphi), \cdots, (G_{t}, \varphi)$ are the connected components of $(G, \varphi)$, then $r(G, \varphi)=\sum_{i=1}^{t}r(G_{i}, \varphi)$.

\item  $r(G, \varphi)\geq 0$ with equality if and only if $(G, \varphi)$ is an empty graph.
\end{enumerate}
\end{lemma}

If a cycle in $G$ consists of edges
$v_1v_2,v_2v_3,\cdots,v_{n}v_1$, it
has two directions:
$v_{1}v_{2} \cdots v_{n}v_{1}$
and $v_{1}v_{n}v_{n-1} \cdots v_{2}v_{1}$.
In the following, assume that each cycle has an order.
If $C$ is the order cycle $v_{1}v_{2} \cdots v_{n}v_{1}$,
then $C^*$ is the order cycle
$v_{1}v_{n}v_{n-1} \cdots v_{2}v_{1}$.
We call $C^*$ the {\it dual order cycle} of $C$.
For any order cycle $C_n=v_{1}v_{2} \cdots v_{n}v_{1}$,
define
$$
\varphi(C_n)=
\prod_{i=1}^n \varphi(v_iv_{i+1}).
$$
By the definition of $\varphi$,
$\varphi(C_n)$ and  $\varphi(C^*_n)$
have the following relation.

\begin{lemma}\label{le2-1}
Let $C_n:v_{1}v_{2} \cdots v_{n}v_{1}$
be an order cycle in a complex unit gain graph
$(G,\varphi)$.
Then $\varphi(C_n)$ and $\varphi(C^*_n)$
are  conjugate numbers
in $\mathbb{T}$.
\end{lemma}

Each order cycle $C_n$ is in one of the five types defined below:
$$
\left\{
    \begin{array}{ll}
      \rm{Type~A},
&\hbox{if $n$ is even and $\varphi(C_{n})=(-1)^{n/2}$;} \\
      \rm{Type~B}, & \hbox{if $n$ is even
      and $\varphi(C_{n})\ne (-1)^{n/2}$}; \\
      \rm{Type~C}, & \hbox{if $n$ is odd and
$Re((-1)^{\frac{n-1}{2}}\varphi(C_{n}, \varphi))>0$;} \\
      \rm{Type~D}, & \hbox{if $n$ is odd and $Re((-1)^{\frac{n-1}{2}}\varphi(C_{n}, \varphi))<0$;} \\
      \rm{Type~E}, & \hbox{if $n$ is odd and
$Re(\varphi(C_{n}, \varphi))=0$,}
    \end{array}
  \right.
$$
where $Re(z)$ is the real part of a complex number $z$.

By the above definition of types
and Lemma~\ref{le2-1},
each pair of dual order cycles $C$ and $C^*$ are in the same type.

\begin{lemma}\label{le2-2}
For any order cycle $C$ in a complex unit gain graph
$(G,\varphi)$,
$C$ and $C^*$ are in the same type.
\end{lemma}

\begin{proof}
By Lemma~\ref{le2-1}, $Re(\varphi(C))=Re(\varphi(C^*))$.
Thus the conclusion holds.
\end{proof}

By Lemma~\ref{le2-2},
whenever the type of a cycle is mentioned,
it is not necessary to know its order.

If $G$ is a cycle $C_n$, the number of positive
(resp. negative) eigenvalues of $(G,\varphi)$
has been determined.

\begin{lemma} \label{L12}{\rm\cite{YGH}}
Let $(C_{n}, \varphi)$ be a complex unit gain cycle of order $n$. Then
$$(p^{+}(C_{n}, \varphi), n^{-}(C_{n}, \varphi))=\left\{
             \begin{array}{ll}
               (\frac{n-2}{2}, \frac{n-2}{2}), & \hbox{if $C_n$ is of \rm{Type~A};} \\
               (\frac{n}{2}, \frac{n}{2}), & \hbox{if $C_n$ is of \rm{Type~B};} \\
               (\frac{n+1}{2}, \frac{n-1}{2}), & \hbox{if $C_n$ is of \rm{Type~C};} \\
               (\frac{n-1}{2}, \frac{n+1}{2}), & \hbox{if $C_n$ is of \rm{Type~D};} \\
               (\frac{n-1}{2}, \frac{n-1}{2}), & \hbox{if $C_n$ is of \rm{Type~E}.}
             \end{array}
           \right.
$$
\end{lemma}

\begin{lemma} \label{L1500}{\rm\cite{YGH}}
Let $(T, \varphi)$ be an acyclic complex unit gain graph.
Then $r(T, \varphi)= r(T)$.
\end{lemma}

It is well known that $r(T)=2m(T)$ for an acyclic graph $T$ in \cite{CVE}. Then we have Lemma \ref{L15}.

\begin{lemma} \label{L15}
Let $(T, \varphi)$ be an acyclic complex unit gain graph.
Then $r(T, \varphi)= r(T)=2m(T)$.
\end{lemma}

\begin{lemma} \label{L13}{\rm\cite{YGH}}
Let $y$ be a pendant vertex of a complex unit gain graph $(G, \varphi)$ and $x$ is the neighbour of $y$.
Then $r(G, \varphi)=r((G, \varphi)-  \{ x, y \} )+2$.
\end{lemma}

\begin{lemma} \label{L14}{\rm\cite{YGH}}
Let $x$ be a vertex of a complex unit gain graph $(G, \varphi)$.
Then $r(G, \varphi)-2 \leq r((G, \varphi)-x) \leq r(G, \varphi)$.
\end{lemma}

\begin{lemma} \label{L17}{\rm\cite{LSC}}
Let $G$ be a simple undirected graph.
Then $m(G)-1 \leq m(G-v) \leq m(G)$ for any vertex $v \in V(G)$.
\end{lemma}

\begin{lemma} \label{L18}{\rm\cite{WDY}}
Let $G$ be a graph obtained by joining a vertex of an even cycle $C$ by an edge to a vertex of a connected graph $H$. Then $m(G)=m(C)+m(H)$.
\end{lemma}

\begin{lemma} \label{L19}{\rm\cite{WDY}}
Let $x$ be a pendant vertex of $G$ and $y$ be the neighbour of $x$. Then $m(G)=m(G-y)+1=m(G- \{ x, y \})+1$.
\end{lemma}

\begin{lemma} \label{L22}{\rm\cite{RULA}}
Let $G$ be a graph with at least one cycle. Suppose that all cycles of $G$
are pairwise-disjoint and each cycle is odd, then $m(T_{G})=m(G-O(G))$ if and only if
$m(G)=\sum_{C \in \mathscr{L}(G)}m(C)+m(G-O(G))$, where $\mathscr{L}(G)$ denotes the set of all cycles in $G$
and $O(G)$ is the set of vertices in cycles of $G$.
\end{lemma}

\begin{lemma} \label{L24}{\rm\cite{RULA}}
Let $G$ be a connected graph without pendant vertices and $c(G) \geq 2$.
Suppose that for any vertex $u$ on a cycle of $G$, $c(G-u) \geq c(G)-2$. Then there are at most
$c(G)-1$ vertices of $G$ which are not covered by its maximum matching.
\end{lemma}


\begin{lemma} \label{L23}{\rm\cite{WDY}}
Let $G$ be a graph with $x \in V(G)$. Then

{\em(i)} $c(G)=c(G-x)$ if $x$ lies outside any cycle of $G$;

{\em(ii)} $c(G-x) \leq c(G)-1$ if $x$ lies on a cycle of $G$;

{\em(iii)} $c(G-x) \leq c(G)-2$ if $x$ is a common vertex of distinct cycles of $G$.
\end{lemma}







\begin{lemma} \label{L20}{\rm\cite{WDY}}
Let $T$ be a tree with at least one edge. Then

{\em(i)} $r(T_{1})< r(T)$, where $T_{1}$ is the subtree obtained from $T$ by deleting all the pendant vertices of $T$.

{\em(ii)} If $r(T-W)= r(T)$ for a subset $W$ of $V(T)$, then there is a pendant vertex $v$ of $T$ such that $v \notin W$.
\end{lemma}

By Lemmas \ref{L1500} and \ref{L20}, we have the following lemma.

\begin{lemma} \label{L21}
Let $(T, \varphi)$ be a complex unit gain tree with at least one edge. Then

{\em(i)} $r(T_{1}, \varphi)< r(T, \varphi)$, where $(T_{1}, \varphi)$ is the subtree obtained from $(T, \varphi)$ by deleting all the pendant vertices of $(T, \varphi)$.

{\em(ii)} If $r((T, \varphi)-W)= r(T, \varphi)$ for a subset $W$ of $V(T)$, then there is a pendant vertex $v$ of $(T, \varphi)$ such that $v \notin W$.
\end{lemma}

\section{Proof of Theorem \ref{T30}}

Lemma~\ref{le3-1} can be verified easily.

\begin{lemma}\label{le3-1}
Let $(G,\varphi)$ be a complex unit gain graph.
If $G$ does not contain any elementary spanning subgraph,
then $\det (H(G,\varphi))=0$.
\end{lemma}

\noindent
{\bf The proof of Theorem \ref{T30}.}
Firstly, we prove that $r(G, \varphi) \leq 2m(G)+c(G)$.
Let $V(G)=\{ v_{1}, v_{2}, \cdots, v_{n} \}$ and $P_{(G, \varphi)}(\lambda)=\lambda^{n}+a_{1} \lambda^{n-1}+ \cdots + a_{n}$ be the characteristic polynomial of $H(G, \varphi)$, where $H(G, \varphi)$ is the adjacent matrix of $(G, \varphi)$.
Observe that $r(G, \varphi) \leq 2m(G)+c(G)$ if
and only if $a_k=0$ for all $k>2m(G)+c(G)$.

Note that $(-1)^{k}a_{k}$ is the sum of all $k\times k$
principal minors of $H(G, \varphi)$,
where a $k\times k$
principal minor of $H(G, \varphi)$
is the determinant of the Hermitian adjacency matrix of
$(G[S], \varphi)$ for some $S\subseteq V$ with $|S|=k$.

Observe that for each elementary spanning subgraph $F$ of
$G[S]$, $k\le m(F)+c(F)$ holds, where the equality holds
if and only if each component of $F$ is either $K_2$
or an odd cycle.
Clearly, $m(F)+c(F)\le m(G)+c(F)$, implying that
when $|S|>2m(G)+c(G)$,
$G[S]$ does not have any each elementary spanning subgraph.
Thus, by Lemma~\ref{le3-1},
if $k>2m(G)+c(G)$, then $a_k=0$.
Thus, $r(G, \varphi) \leq 2m(G)+c(G)$.

Next,  we argue by induction on $c(G)$ to show that $2m(G)-2c(G) \leq r(G, \varphi) $.
If $c(G)=0$, then $r(G,\varphi)=2m(G)$ holds
by Lemma \ref{L15}.
Now assume that $c(G) \geq 1$.
Let $x$ be a vertex on some cycle of $G$ and
$G'=G-x$.
Let $G_1,\cdots,G_t$ be the components of $G'$.
By Lemma \ref{L23}, we have
\begin{equation} \label{E1}
\sum_{i=1}^{t}c(G_{i})=c(G') \leq c(G)-1.
\end{equation}
By Lemmas \ref{L16}, \ref{L14} and \ref{L17}, we have
\begin{equation} \label{E2}
\sum\limits_{i=1}^{t}r(G_{i}, \varphi)=r(G', \varphi) \leq r(G, \varphi)
\end{equation}
and
\begin{equation} \label{E3}
\sum_{i=1}^{t}m(G_{i})=m(G')\geq m(G)-1.
\end{equation}
Since $c(G') \leq c(G)-1$, by the induction hypothesis,
for each $j \in \{ 1, 2, \cdots, t \}$, 
\begin{equation} \label{E4}
r(G_{j}, \varphi) \geq 2m(G_{j}) -2c(G_{j}).
\end{equation}
Combining (\ref{E1}), (\ref{E2}),  (\ref{E3}) and (\ref{E4}), one has that
\begin{eqnarray} \label{E0}
r(G, \varphi) & \ge & \sum\limits_{i=1}^{t}r(G_{i}, \varphi)
\\ \nonumber
& \ge & \sum\limits_{i=1}^{t} [2m(G_{j}) -2c(G_{j})]
\\ \nonumber
& \geq & 2(m(G) - 1) - 2(c(G) - 1)
\\ \nonumber
&= &2m(G)-2c(G).
\end{eqnarray}
This completes the proof of Theorem \ref{T30}.
$\square$

\section{Proof of Theorem \ref{T50}.}
A complex unit gain graph $(G, \varphi)$ is said to be {\bf lower-optimal} if $r(G, \varphi)=2m(G)-2c(G)$,
or equivalently, the complex unit gain graphs attain the lower bound in Theorem \ref{T30}.
In this section, the proof for Theorem \ref{T50} is provided.
Firstly, we introduce some useful lemmas which will be used to prove the main result of this section.

The following Lemma \ref{L50} can be derived from Lemma \ref{L12} directly.

\begin{lemma} \label{L50}
The complex unit gain cycle $(C_{q}, \varphi)$ is lower-optimal if and only if $\varphi(C_{q}, \varphi)=(-1)^{\frac{q}{2}}$ and $q$ is even.
\end{lemma}

\begin{lemma} \label{L51}
Let $(G, \varphi)$ be a complex unit gain graph with connected components
$(G_{1}, \varphi), (G_{2}, \varphi),\\ \cdots, (G_{k}, \varphi)$. Then $(G, \varphi)$ is lower-optimal if and only if $(G_{i}, \varphi)$ is lower-optimal for each
$i \in \{  1, 2, \cdots, k \}$.
\end{lemma}

\begin{proof} (Sufficiency.) Since $(G_{i}, \varphi)$ is lower-optimal for each
$i \in \{  1, 2, \cdots, k \}$, one has that
$$r(G_{i}, \varphi)= 2m(G_{i})-2c(G_{i}).$$
By Lemma \ref{L16}, we have
\begin{eqnarray*}
r(G, \varphi)&=&\sum\limits_{j=1}^{k}r(G_{j}, \varphi)\\
&=&\sum\limits_{j=1}^{k}[2m(G_{i})-2c(G_{i})]\\
&=&2m(G)-2c(G).
\end{eqnarray*}

(Necessity.) By contradiction, suppose that there is a connected component of $(G, \varphi)$, say $(G_{1}, \varphi)$, which is not lower-optimal. By Theorem \ref{T30}, one has that
$$r(G_{1}, \varphi)> 2m(G_{1})-2c(G_{1})$$
and for each $ j \in \{ 2, 3, \cdots, k \}$, we have
$$r(G_{j}, \varphi) \geq 2m(G_{j})-2c(G_{j}).$$
By Lemma \ref{L16}, one has that
$$r(G, \varphi) > 2m(G)-2c(G),$$
which is a contradiction.
\end{proof}

\begin{lemma} \label{L520}
Let $(G, \varphi)$ be a complex unit gain graph and $u$ be a vertex of $(G, \varphi)$ lying on a complex unit gain cycle. If $(G, \varphi)$ is lower-optimal, then
each of the following holds.
\\
{\em(i)} $r(G, \varphi)=r((G, \varphi)-u)$;
\\
{\em(ii)} $(G, \varphi)-u$ is lower-optimal;
\\
{\em(iii)} $c(G)=c(G-u)+1$;
\\
{\em(iv)} $m(G)=m(G-u)+1$;
\\
{\em(v)} $u$ lies on just one complex unit gain cycle of $(G, \varphi)$ and $u$ is not a quasi-pendant vertex of $(G, \varphi)$.
\end{lemma}

\begin{proof} Since $(G, \varphi)$ is lower-optimal, all inequalities in (\ref{E0}) in the proof of Theorem \ref{T30} must be equalities, and so Lemma \ref{L520} (i)-(iv) are derived.

As for (v), by Lemma \ref{L520} (iii) and Lemma \ref{L23}, one has that $u$ lies on just one
complex unit gain cycle of $(G, \varphi)$.
Suppose to the contrary that $u$ is a quasi-pendant vertex which is adjacent to a vertex $v$. Then $v$ is
an isolated vertex in $(G, \varphi)-u$ and $r((G, \varphi)-u)=r((G, \varphi)- \{ u, v  \})$.
By Lemma \ref{L13}, we have
$$r((G, \varphi)-u)=r(G, \varphi)-2,$$
which is a contradiction to (i).
\end{proof}

\begin{lemma} \label{L53}
Let $(G, \varphi)$ be a connected complex unit gain graph with $m(T_{G})=m(G-O(G))$, where $O(G)$ is the set of vertices in cycles of $G$.
Then every vertex lying on a complex unit gain cycle can not be a quasi-pendant vertex of $(G, \varphi)$.
\end{lemma}
\begin{proof}
Suppose to the contrary that there exists a quasi-pendant vertex $u$ lying on a complex unit gain cycle of
$(G, \varphi)$.
Let $v$ be the pendant vertex which is adjacent to $u$ and $M$ be a maximum matching of
$G-O(G)$. Then $v$ is a pendant vertex of $G-O(G)$ and $M \cup \{uv \}$ is also a matching of $T_{G}$.
Thus, $m(T_{G}) \geq m(G-O(G))+1$. This  contradiction proves the lemma.
\end{proof}

\begin{lemma} \label{L56}
Let $(G, \varphi)$ be a complex unit gain graph which contains a pendant vertex $u$ with its unique neighbour $v$.
Let $(G', \varphi)=(G, \varphi)- \{ u, v\} $. If $(G, \varphi)$ is lower-optimal, then $(G', \varphi)$ is also lower-optimal.
\end{lemma}
\begin{proof} Since $(G, \varphi)$ is lower-optimal, one has that\
$$r(G, \varphi)=2m(G)-2c(G).$$
Moreover, by Lemmas \ref{L13}, \ref{L19}, \ref{L520} and \ref{L23}, we have
$$r(G', \varphi)=r(G, \varphi)-2, m(G')=m(G)-1, c(G)=c(G').$$
Then, $r(G', \varphi)=2m(G')-2c(G')$ can be obtained, and so $(G', \varphi)$ is also lower-optimal.
\end{proof}

\begin{lemma} \label{L540}
Let $(G, \varphi)$ be a connected complex unit gain unicyclic graph of order $n$ whose unique complex unit gain cycle is $(C_{l}, \varphi)$.
If $r(G, \varphi)=2m(G)-2$, then $\varphi(C_{l}, \varphi)=(-1)^{\frac{l}{2}}$ and $l$ is even.
\end{lemma}
\begin{proof}
Let $V(G, \varphi)=\{ v_{1}, v_{2}, \cdots, v_{n} \}$ and $P_{(G, \varphi)}(\lambda)=|\lambda I_{n}-H(G, \varphi)|=\lambda^{n}+a_{1} \lambda^{n-1}+ \cdots + a_{n}$
be the characteristic polynomial of $H(G, \varphi)$ and $m=m(G)$.
It can be checked that the number $(-1)^{k}a_{k}$ is the sum of all
principal minors of $H(G, \varphi)$ with $k$ rows and $k$ columns. Where, each such minor is the determinant of the Hermitian-adjacency matrix of an induced subgraph $(G_{0}, \varphi)$ of $(G, \varphi)$ with $k$ vertices.
By similar method with the proof of the inequality on the right of Theorem \ref{T30},
one has that each non-vanishing term in
the determinant expansion gives rise to an elementary complex unit gain subgraph $(G'_{0}, \varphi)$ of $(G_{0}, \varphi)$ with $|V(G'_{0}, \varphi)|=|V(G_{0}, \varphi)|=k$. Thus, $(G'_{0}, \varphi)$ is a spanning elementary subgraph of $(G_{0}, \varphi)$.

The sign of a permutation $\pi$ is $(-1)^{N_{e}}$, where $N_{e}$ is the number of even cycles (i.e., cycles with even length) in $\pi$. If there are $c_{l}$ cycles of length $l$, then the equation
$\sum l c_{l}=|V(G'_{0}, \varphi)| $ shows that the number $N_{o}$ of odd cycles is congruent to $|V(G'_{0}, \varphi)|$ modulo 2. Hence,
$$|V(G'_{0}, \varphi)|-(N_{o}+N_{e}) \equiv N_{e}\, ({\rm{mod}} \,2),$$
so the sign of $\pi$ is equal to $(-1)^{N_{e}}$.

Each spanning elementary subgraph $(G'_{0}, \varphi)$ gives rise to several permutations $\pi$ for which
the corresponding term in the determinant expansion does not vanish. If $(G'_{0}, \varphi)$ contains a
complex unit gain cycle-component, then the number of
such $\pi$ arising from a given $(G'_{0}, \varphi)$ is 2, since for each complex unit gain cycle-component in $(G'_{0}, \varphi)$ there are two ways of choosing the corresponding cycle in $\pi$. Furthermore, if for some
direction of a permutation $\pi$, a complex unit gain cycle-component has value $\boldsymbol{i}$ (or $-\boldsymbol{i}$), then for the
other direction the complex unit gain cycle-component has value $-\boldsymbol{i}$ (or $\boldsymbol{i}$) and vice versa. Thus, they
offset each other in the summation. Similarly, if for some direction of a permutation $\pi$,
a complex unit gain cycle-component has value 1 (or -1), then for the other direction the complex unit gain
cycle-component has value 1 (or -1) too.
If for some direction of a permutation $\pi$,
a complex unit gain cycle-component has value $a+b\boldsymbol{i}$, then for the other direction the complex unit gain
cycle-component has value $a-b\boldsymbol{i}$ (note that $a+b\boldsymbol{i}$ is also a unit complex number and $a=Re(\varphi(C_{l}, \varphi))$).
Moreover, each complex unit gain edge-component has value 1.

Since $r(G, \varphi)=2m(G)-2$, $a_{2m}=0$. It can be checked that $l$ is even.
Assume for a contradiction that $l$ is odd. Then each elementary complex unit gain subgraph with order $2m$ contains only edges as its components and the sign of every permutation $\pi$ is $(-1)^{m}$. Then
$$a_{2m}=\sum\limits_{M \in \mathbb{M}}(-1)^{m}=(-1)^{m} |\mathbb{M}|\neq 0,$$
where $\mathbb{M}$ is the set of all maximum matchings of $(G, \varphi)$,
a contradiction. Thus, $l$ is even.

As $(G, \varphi)$ is a connected complex unit gain unicyclic graph, from the above analysis one has that some elementary complex unit gain subgraphs of $(G, \varphi)$ with order $2m$ contain the cycle $(C_{l}, \varphi)$ and $\frac{2m-l}{2}$ edge-components (each component is an even cycle in this case), and some elementary complex unit gain subgraphs with order $2m$ contain only $m$ edge-components.
Thus, we have
\begin{eqnarray*}
a_{2m}&=&\sum\limits_{(U, \varphi) \in \mathscr{U}_{2m}} (-1)^{p(U, \varphi)} \cdot 2^{1}\cdot Re(\varphi(C_{l}, \varphi))
+\sum\limits_{M \in \mathbb{M}}(-1)^{m}\\
&=&\sum\limits_{(U, \varphi) \in \mathscr{U}_{2m}} (-1)^{1+\frac{2m-l}{2}} \cdot 2^{1} \cdot Re(\varphi(C_{l}, \varphi))
+\sum\limits_{M \in \mathbb{M}}(-1)^{m}\\
&=&(-1)^{m} \{ |\mathbb{M}|+2|\mathscr{U}_{2m}|(-1)^{\frac{l+2}{2}}\cdot Re(\varphi(C_{l}, \varphi))  \},
\end{eqnarray*}
where $\mathscr{U}_{2m}$ denotes the set of all elementary subgraphs of order $2m$ which contains $(C_{l}, \varphi)$ as its connected component, $\mathbb{M}$ is the set of all maximum matchings of $(G, \varphi)$ and
$p(U, \varphi)$ is the number of components of $(U, \varphi)$, respectively.
By the fact that $a_{2m}=0$, then
$$|\mathbb{M}|+2|\mathscr{U}'_{2m}|(-1)^{\frac{l+2}{2}}\cdot Re(\varphi(C_{l}, \varphi)) =0.$$
It can be checked that $|\mathbb{M}| \geq 2|\mathscr{U}'_{2m}|$, as $2|\mathscr{U}'_{2m}|$ matchings of $(G, \varphi)$ of size $2m$ can be found by using the two matching in the cycle $(C_{l}, \varphi)$.
On the other hand $|Re(\varphi(C_{l}, \varphi))| \leq 1$. Thus, $(-1)^{\frac{l+2}{2}}\cdot Re(\varphi(C_{l}, \varphi))=-1$ and $\varphi(C_{l}, \varphi)=(-1)^{\frac{l}{2}}$.

This completes the proof.
\end{proof}

\begin{lemma} \label{L54}
Let $(G, \varphi)$ be a connected complex unit gain unicyclic graph whose unique complex unit gain cycle is $(C_{l}, \varphi)$. Then the following conditions are equivalent.

{\em(i)} $r(G, \varphi)=2m(G)-2$;

{\em(ii)} $\varphi(C_{l}, \varphi)=(-1)^{\frac{l}{2}}$ and $l$ is even, and $m(T_{G})=m(G-O(G))$, where $O(G)$ is the set of vertices in cycles of $G$.
\end{lemma}
\begin{proof}
{\bf(ii) $\Rightarrow$ (i).}  We argue by induction on the order of $T_{G}$.
If $|V(T_{G})|=1$, then $(G, \varphi) \cong (C_{l}, \varphi)$.
Since $\varphi(C_{l}, \varphi)=(-1)^{\frac{l}{2}}$ and $l$ is even, by Lemma \ref{L12}, we have $r(G, \varphi)=2m(G)-2$.

Next, one can suppose that $|V(T_{G})| \geq 2$, then there is a pendant vertex $u$ of $T_{G}$ which is also a pendant vertex of $(G, \varphi)$. Let $v$ be the unique neighbour of $u$.
By Lemma \ref{L53}, $v$ is not on any cycle of $(G, \varphi)$. Let $(G_{0}, \varphi)=(G, \varphi)- \{ u, v \}$
and $(G_{1}, \varphi), (G_{2}, \varphi), \cdots, (G_{k}, \varphi)$ be all the connected components of $(G_{0}, \varphi)$.
It is routine to check that
$$T_{G_{0}}=T_{G}-\{ u, v \}, G_{0}-O(G_{0})=G-O(G)-\{ u, v \}$$
and
$$m(T_{G_{0}})=m(G_{0}-O(G_{0})).$$
By Lemmas \ref{L13} and \ref{L19}, one has that
$$r(G, \varphi)=r(G_{0}, \varphi)+2, m(G)=m(G_{0})+1.$$
Without loss of generality, assume that $(G_{1}, \varphi)$ is the connected component which contains the
unique cycle $(C_{l}, \varphi)$. Then $(G_{j}, \varphi)$ is a tree for each $j \in \{ 2, 3, \cdots, k \}$, and
 $m(T_{G_{0}})=m(G_{0}-O(G_{0}))$ implies $m(T_{G_{j}})=m(G_{j}-O(G_{j}))$.
Since $(G_{1}, \varphi)$ is a connected unicyclic graph and $|V(T_{G_{1}})| < |V(T_{G})|$,
by induction hypothesis, one has that
$$r(G_{1}, \varphi)=2m(G_{1})-2.$$
By Lemma \ref{L15}, we have
$r(G_{j}, \varphi)=2m(G_{j})$ for each $j \in \{ 2, 3, \cdots, k\}$.
Then, by Lemma \ref{L16}, we have
\begin{eqnarray*}
r(G, \varphi)&=&r(G_{0}, \varphi)+2\\
&=&r(G_{1}, \varphi)+\sum_{j=2}^{k}r(G_{j}, \varphi)+2\\
&=&2m(G_{1})-2+2\sum_{j=2}^{k}m(G_{j})+2\\
&=&2m(G_{0})\\
&=&2m(G)-2.
\end{eqnarray*}

{\bf(i) $\Rightarrow$ (ii).} By Lemma \ref{L540} and the condition $r(G, \varphi)=2m-2$, one has that $\varphi(C_{l}, \varphi)=(-1)^{\frac{l}{2}}$ and $l$ is even.

We show $m(T_{G})=m(G-O(G))$ by induction on $|V(T_{G})|$.
Since $(G, \varphi)$ contains a cycle, $|V(T_{G})| > 0$.  If $|V(T_{G})|=1$, then $(G, \varphi) \cong C_{l}$ and $m(T_{G})=m(G-O(G))=0$.
Next, one can assume that $|V(T_{G})|\geq 2$. Then there exists a pendant vertex
$u$ of $T_{G}$ which is also a pendant vertex of $(G, \varphi)$. Let $v$ be the unique neighbour of $u$.
By Lemma \ref{L520} (v), we have $v$ is not on any cycle of $(G, \varphi)$.
Denote $(G', \varphi)=(G, \varphi)- \{ u, v \}$. Let $(G'_{1}, \varphi), (G'_{2}, \varphi), \cdots, (G'_{k}, \varphi)$ be all connected components of $(G', \varphi)$. Without loss of generality, assume that $(G'_{1}, \varphi)$ contains the unique cycle $(C_{l}, \varphi)$.
By Lemmas \ref{L13} and \ref{L19}, one has that
$$r(G', \varphi)=2m(G')-2.$$
Since $|V(T_{G'})| < |V(T_{G})|$, by the induction hypothesis, one has that
$$m(T_{G'_{1}})=m(G'_{1}-O(G'_{1})).$$
It is routine to check that $c(G)=c(G')=1$, and for each $j \in \{ 2, 3, \cdots, k \}$,
$$T_{G'_{j}} \cong G'_{j}-O(G'_{j}).$$
Thus, by Lemma \ref{L19}, we have
$$m(T_{G})=m(T_{G'})+1=m(G'-O(G'))+1=m(G-O(G)).$$

The result follows.
\end{proof}

\begin{lemma} \label{L55}
Let $(G, \varphi)$ be a complex unit gain graph obtained by joining a vertex $x$ of a complex unit gain cycle, say $(O, \varphi)$, by an edge to a vertex $y$ of a connected complex unit gain graph $(K, \varphi)$.
If $(G, \varphi)$ is lower-optimal, then the following properties hold for $(G, \varphi)$.

{\em(i)} Every complex unit gain cycle $(C_{l}, \varphi)$ of $(G, \varphi)$ satisfies $\varphi(C_{l}, \varphi)=(-1)^{\frac{l}{2}}$ and $l$ is even.

{\em(ii)} The edge $xy$ does not belong to any maximum matching of $(G, \varphi)$.

{\em(iii)} Each maximum matching of $K$ saturates $y$.

{\em(iv)} $m(K+x)=m(K)$.

{\em(v)} $(K, \varphi)$ is lower-optimal.

{\em(vi)} Let $(G', \varphi)$ be the induced complex unit gain subgraph of $(G, \varphi)$ with vertex set $V(K)\cup \{ x\}$. Then $(G', \varphi)$ is also lower-optimal.
\end{lemma}
\begin{proof}
{\bf  (i).}  We argue by induction on $c(G)$. Since $(G, \varphi)$ contains cycle, $c(G) \geq 1$.   If $c(G)=1$, then $(G, \varphi)$ is a complex unit gain unicyclic graph. The result follows from Lemma \ref{L54} immediately. Next, one can suppose that $c(G) \geq 2$. Then $(K, \varphi)$
contains at least one cycle. Let $u$ be a vertex lying on some cycle of
$(K, \varphi)$ and $(G_{0}, \varphi)=(G, \varphi)-u$. By Lemmas \ref{L520} and \ref{L23}, we have $(G_{0}, \varphi)$ is lower-optimal and  $c(G_{0}) < c(G)$.
By induction hypothesis, one has that each cycle in $G_{0}$, including $(O, \varphi)$, satisfies (i). By a similar discussion as for $(G, \varphi)-x$, we can show that all the cycles in $(K, \varphi)$ satisfy (i). This completes the proof of (i).

{\bf  (ii).} Suppose to the contrary that there is a maximum matching $M$ of $(G, \varphi)$ containing $xy$.
By (i), one has that $(O, \varphi)$ is an even cycle. Then there exists a vertex $w \in V(O)$ such that $w$ is not saturated by $M$. Then we have $m(G)=m(G-w)$, a contradiction to Lemma \ref{L520} (iv).

{\bf  (iii).} By Lemma \ref{L18}, we have $m(G)=m(K)+m(O)$. Let $M_{1}$ and $M_{2}$ be the maximum matchings of $O$ and $K$, respectively. Then $M_{1} \cup M_{2}$ is a maximum matching of $(G, \varphi)$.
Suppose to the contrary that there exists a maximum matching of $K$ fails to saturate $y$.
Then we obtain a maximum matching $M'_{1} \cup M_{2}$ of
$(G, \varphi)$ which contains $xy$, where $M'_{1}$ is obtained from $M_{1}$ by replacing the edge in $M_{1}$ which saturates $x$ with $xy$, a contradiction to (ii).

{\bf  (iv).} Since each maximum matching of $K$ saturates $y$, it is routine to check that $m(K+x)=m(K)$.

{\bf  (v).} By Lemma \ref{L520} (ii), $(G, \varphi)-x$ is lower-optimal. Then (v) immediately follows from Lemma \ref{L51}.

{\bf  (vi).} Suppose that $O=xx_{2}x_{3} \cdots x_{2s}x$. Since $(G, \varphi)$ is lower-optimal, by Lemma \ref{L520} (ii), one has that $(G_{1}, \varphi)= (G, \varphi)-x_{2}$ is also lower-optimal.
Obviously, $x_{3}$ and $x_{4}$ are pendant vertex and quasi-pendant vertex of $(G_{1}, \varphi)$, respectively.
By Lemma \ref{L56}, one has that $(G_{2}, \varphi)= (G_{1}, \varphi)- \{  x_{3}, x_{4} \}$  is also lower-optimal.
Repeating such process (deleting a pendant vertex and a quasi-pendant vertex), after $s-1$ steps, the result graph is $(G, \varphi)- \{x_{2}, x_{3}, \cdots, x_{2s}  \}=(G', \varphi)$. By Lemma \ref{L56}, $(G', \varphi)$ is also lower-optimal.
\end{proof}

\begin{lemma} \label{L57}
Let $(G, \varphi)$ be a connected complex unit gain graph. If $(G, \varphi)$ is lower-optimal, then there exists a maximum matching $M$ of $(G, \varphi)$ such that $M \cap \mathscr{F}(G) = \emptyset$, where $\mathscr{F}(G)$ denotes the set of edges of $(G, \varphi)$, each of which has one endpoint in a cycle and the other endpoint outside the cycle.
\end{lemma}
\begin{proof}
We argue by induction on $|V(T_{G})|$. If $|V(T_{G})|=1$, then $(G, \varphi)$
is either a complex unit gain cycle or an isolated vertex and the conclusion holds trivially.
Then one can suppose that
$|V(T_{G})| \geq 2$, and so $T_{G}$ has at least one pendant vertex, say $u$.

Suppose $u$ is also a pendant vertex of $(G, \varphi)$.
Let $v$ be the unique neighbour of $u$ in $(G, \varphi)$ and $(G_{0}, \varphi)=(G, \varphi)- \{ u, v \}$.
By Lemmas \ref{L520} and \ref{L56}, one has that $v$ is not on any cycle of $(G, \varphi)$ and $(G_{0}, \varphi)$ is also lower-optimal. Let $(G_{1}, \varphi), (G_{2}, \varphi), \cdots, (G_{k}, \varphi)$ be all connected components of $(G_{0}, \varphi)$. Then it follows from
Lemma \ref{L51} that $(G_{j}, \varphi)$ is lower-optimal for each $j \in \{ 1,2 \cdots, k \}$. Applying induction hypothesis to
$(G_{j}, \varphi)$ yields that there exists a maximum matching $M_{j}$ of $G_{j}$ such that
$M_{j} \cap \mathscr{F}(G)= \emptyset $ for each $j \in \{ 1 ,2 \cdots, k \}$. Let $M=(\cup_{j=1}^{k}M_{j}) \cup \{uv\}$. Then it can be checked that
$M$ is a maximum matching of $(G, \varphi)$ which satisfies $M \cap \mathscr{F}(G) = \emptyset$.

If $u$ lies on some complex unit gain cycle of $(G, \varphi)$, then $(G, \varphi)$ has a pendant cycle, say $(C', \varphi)$. Let $(K, \varphi)=(G, \varphi)-(C', \varphi)$. By Lemma \ref{L55}, each cycle of $(G, \varphi)$ is even and $(K, \varphi)$ is lower-optimal.
Applying the induction hypothesis to $(K, \varphi)$ implies that there exists a maximum matching $M_{0}$
of $(K, \varphi)$ such that $M_{0} \cap \mathscr{F}(K)  = \emptyset$. Let $M'_{0}$ be a maximum matching of $(C', \varphi)$.
By Lemma \ref{L18}, it is routine to verify that $M =M _{0} \cup M'_{0}$
is a maximum matching of $(G, \varphi)$ satisfying $M \cap \mathscr{F}(G)  = \emptyset$.
This completes the proof.
\end{proof}

Now, we give the proof of the main result of this section.

\noindent
{\bf The proof of Theorem \ref{T50}.}
(Sufficiency.) We show $r(G, \varphi)=2m(G)-2c(G)$ by induction on $|V(T_{G})|$. If $|V(T_{G})|=1$, then $(G, \varphi)$ is either a complex unit gain cycle or an isolated vertex. By Lemma \ref{L12}, the result holds in this case.

Therefore one can assume that
$|V(T_{G})| \geq 2$.  Since $m(T_{G})=m(G-O(G))$, by Lemma \ref{L15}, one has
that $r(T_{G})=r(G-O(G))$.
By Lemma \ref{L21} and $T_{G}$ is an acyclic graph, one has that
$(G, \varphi)$ has at least one pendant vertex, say $u$.
Let $v$ be the unique neighbour of $u$ in $(G, \varphi)$. By Lemma \ref{L53}, $v$ does not lie on any cycle of $(G, \varphi)$.
Let $(G_{0}, \varphi)=(G, \varphi)- \{ u, v \}$ and $(G_{1}, \varphi), (G_{2}, \varphi), \cdots, (G_{k}, \varphi)$ be all connected components of $(G_{0}, \varphi)$.
By Lemma \ref{L19}, we have $m(G)=m(G_{0})+1$.
It is routine to check that $v$ is also a vertex of $T_{G}$ (resp. $G-O(G)$)
which is adjacent to $u$ and $T_{G_{0}} = T_{G}- \{ u, v \}$ (resp. $G_{0}-O(G_{0}) = G-O(G)-\{u, v\}$).
Hence, we have
$$m(T_{G})=\sum_{j=1}^{k}m(T_{G_{j}})+1, m(G-O(G))=\sum_{j=1}^{k}m(G_{j}-O(G_{j}))+1.$$

Note that $m(T_{G_{j}}) \geq m(G_{j}-O(G_{j}))$ for each $j \in \{ 1, 2, \cdots, k \}$.
If there exists some $j \in \{ 1, 2, \cdots, k \}$ such that $m(T_{G_{j}}) > m(G_{j}-O(G_{j}))$,
then we have $m(T_{G})  > m(G-O(G))$, a contradiction to (iii). Thus, one has that
$m(T_{G_{j}})=m(G_{j}-O(G_{j}))$ for each $j \in \{ 1, 2, \cdots, k \}$.
Therefore, $(G_{j}, \varphi)$ satisfies (i)-(iii) for each $j \in \{ 1, 2, \cdots, k \}$. Applying the induction hypothesis to $(G_{j}, \varphi)$ yields that for each $j \in \{ 1, 2, \cdots, k \}$,
$$r(G_{j}, \varphi)=2m(G_{j})-2c(G_{j}).$$

Then, one has $r(G, \varphi)=2m(G)-2c(G)$ by the fact that
$$c(G)=c(G_{0})=\sum_{j=1}^{k}c(G_{j}), m(G_{0})=\sum_{j=1}^{k}m(G_{j})$$
and
$$m(G)=m(G_{0})+1=\sum_{j=1}^{k}m(G_{j})+1.$$

(Necessity.) Let $(G, \varphi)$ be a complex unit gain graph satisfying $r(G, \varphi)=2m(G)-2c(G)$. If $(G, \varphi)$ is a complex unit gain acyclic graph, then
(i)-(iii) hold trivially. So one can suppose that $(G, \varphi)$ contains cycles. By Lemma \ref{L520} (v),
(i) follows immediately.

Next, we show (ii) and (iii) by induction on the order $n$ of $(G, \varphi)$.
Since $(G, \varphi)$ contains cycles, $n\geq 3$.
If $n=3$, then $(G, \varphi)$ is a complex unit gain 3-cycle. Moreover, (ii) holds by Lemma \ref{L12} and
(iii) holds by the fact that $m(T_{G})=m(G-O(G))=0$. Suppose that (ii) and (iii) hold for any lower-optimal complex unit gain graph of
order smaller than $n$, and suppose $(G, \varphi)$ is a lower-optimal complex unit gain graph with order $n \geq 4$. If
$|V(T_{G})|=1$, then $(G, \varphi)$ is a complex unit gain cycle. Thus (ii) follows from Lemma \ref{L12} and (iii) follows from the fact that $m(T_{G})=m(G-O(G))=0$.
So, one can suppose that $|V(T_{G})| \geq 2$, then $T_{G}$ has at least one pendant vertex, say $u$. Therefore, it suffices to consider the following two possible cases.

\noindent
{\bf  Case 1.} $u$ is a pendant vertex of $(G, \varphi)$.

Let $v$ be the adjacent vertex of $u$ and $(G', \varphi)=(G, \varphi)-\{ u, v \}$. By Lemmas \ref{L520} and \ref{L56}, $v$ does not lie on any cycle of $(G, \varphi)$ and $(G', \varphi)$ is also lower-optimal. Then it follows from Lemma \ref{L51} that every connected component of $(G', \varphi)$ is lower-optimal. Let $(G'_{1}, \varphi), (G'_{2}, \varphi), \cdots, (G'_{k}, \varphi)$ be all connected components of $(G', \varphi)$. Applying the induction hypothesis to $(G'_{i}, \varphi)$ for each $i \in \{ 1, 2, \cdots, k  \}$ yields:

(a) each cycle $(C_{q}, \varphi)$ of $(G'_{i}, \varphi)$ satisfies $\varphi(C_{q}, \varphi)=(-1)^{\frac{q}{2}}$ and $q$ is even;

(b) $m(T_{G'_{i}})=m(G'_{i}-O(G'_{i}))$.

Assertion (a) implies that each cycle (if any) $(C_{q}, \varphi)$ of $(G, \varphi)$ satisfies $\varphi(C_{q}, \varphi)=(-1)^{\frac{q}{2}}$ and $q$ is even since all cycles of $(G, \varphi)$ belong to $(G', \varphi)$ in this case. Hence, (ii) holds in this case.
Note that $u$ is also a pendant vertex of $T_{G}$ (resp., $G-O(G)$) which is adjacent to $v$ and $T_{G'}=T_{G}-\{ u, v \}$ (resp., $G'-O(G')=G-O(G)-\{ u, v \}$).
By Lemma \ref{L19} and (b), one has that
\begin{eqnarray*}
m(T_{G})&=&m(T_{G'})+1\\
&=&\sum_{j=1}^{k}m(T_{G'_{i}})+1\\
&=&\sum_{j=1}^{k}m(G'_{i}-O(G'_{i}))+1\\
&=&m(G'-O(G'))+1\\
&=&m(G-O(G)).
\end{eqnarray*}

Thus (iii) holds in this case.

\noindent
{\bf  Case 2.} $u$ lies on some pendant complex unit gain cycle of $(G, \varphi)$.

Then, $(G, \varphi)$ contains at least one pendant complex unit gain cycle. By Lemma \ref{L55} (i),
the result (ii) follows immediately.

Next, we just need to prove that $m(T_{G})=m(G-O(G))$.
Let $(C'_{1}, \varphi), (C'_{2}, \varphi), \cdots, (C'_{k}, \varphi)$ be all cycles of $(G, \varphi)$. Without loss of generality,
one can assume that $u$ is the unique vertex of the pendant cycle $(C'_{1}, \varphi)$ with degree 3.
Let $(G_{1}, \varphi)=(G, \varphi)-(C'_{1}, \varphi)$ and $(G_{2}, \varphi)=(G_{1}, \varphi)+x$.
By Lemma \ref{L55} (vi), one has that $(G_{2}, \varphi)$ is lower-optimal.
Since $|V(G_{2}, \varphi)| < |V(G, \varphi)|$, by induction hypothesis to $(G_{2}, \varphi)$, one has that
$$m(T_{G_{2}})=m(G_{2}-O(G_{2})).$$
Hence, by Lemma \ref{L57}, $(G_{2}, \varphi)$ has a maximum matching $M_{2}$ such that $M_{2} \cap \mathscr{F}(G_{2})= \emptyset $, from which it follows that
$$m(G_{2})=m(G_{2}-O(G_{2}))+\frac{\sum_{j=2}^{k}|V(C'_{j})|}{2}.$$
By Lemma \ref{L55} (iv), we have
$$m(G_{1})=m(G_{2}).$$
Moreover, by Lemma \ref{L18}, one has that
$$m(G)=m(C'_{1})+m(G_{1}).$$
By Lemma \ref{L57}, there exists a maximum matching $M$ of $(G, \varphi)$ such
that $M \cap \mathscr{F}(G) = \emptyset $. Consequently,
$$m(G)=m(G-O(G))+\frac{\sum_{j=1}^{k}|V(C'_{j})|}{2}.$$

Note that $T_{G} \cong T_{G_{2}}$.
Thus,
\begin{eqnarray*}
m(T_{G})&=&m(T_{G_{2}})\\
&=&m(G_{2}-O(G_{2}))\\
&=&m(G_{2})-\frac{\sum_{j=2}^{k}|V(C'_{j})|}{2}\\
&=&m(G_{1})-\frac{\sum_{j=2}^{k}|V(C'_{j})|}{2}\\
&=&m(G_{1})+\frac{|V(C'_{1})|}{2}-\frac{\sum_{j=1}^{k}|V(C'_{j})|}{2}\\
&=&m(G)-m(C'_{1})+\frac{|V(C'_{1})|}{2}-\frac{\sum_{j=1}^{k}|V(C'_{j})|}{2}\\
&=&m(G)-\frac{\sum_{j=1}^{k}|V(C'_{j})|}{2}\\
&=&m(G-O(G)).
\end{eqnarray*}
This completes the proof. $\square$

\section{Proof of Theorem \ref{T40}.}
A complex unit gain graph $(G, \varphi)$ is said to be {\bf upper-optimal} if $r(G, \varphi)=2m(G)+c(G)$,
or equivalently, the complex unit gain graphs which attain the upper bound in Theorem \ref{T30}.
In this section, the properties of the complex unit gain graphs which are upper-optimal are characterized,
and the proof of Theorem \ref{T40} is given.

\begin{lemma} \label{L041}
Let $(G, \varphi)$ be a complex unit gain graph and $(G_{1}, \varphi), (G_{2}, \varphi), \cdots, (G_{k}, \varphi)$ be all connected components of $(G, \varphi)$. Then $(G, \varphi)$ is upper-optimal if and only if $(G_{j}, \varphi)$ is upper-optimal for each $j \in \{1, 2, \cdots, k \}$.
\end{lemma}
\begin{proof} (Sufficiency.) For each $i \in \{ 1, 2, \cdots, k \}$, one has that
$$r(G_{i}, \varphi)= 2m(G_{i})+c(G_{i}).$$
Then, by Lemma \ref{L16}, we have
\begin{eqnarray*}
r(G, \varphi)&=&\sum\limits_{j=1}^{k}r(G_{j}, \varphi)\\
&=&\sum\limits_{j=1}^{k}[2m(G_{i})+c(G_{i})]\\
&=&2m(G)+c(G).
\end{eqnarray*}

(Necessity.) Suppose to the contrary that there is a connected component of $(G, \varphi)$, say $(G_{1}, \varphi)$, which is not upper-optimal. By Theorem \ref{T30}, for each $ j \in \{ 2, 3, \cdots, k \}$, one has that
$$r(G_{j}, \varphi) \leq 2m(G_{j})+c(G_{j})$$
and
$$r(G_{1}, \varphi) < 2m(G_{1})+c(G_{1}).$$

Thus, we have
$$r(G, \varphi)=\sum\limits_{j=1}^{k}r(G_{j}, \varphi) <  2m(G)+c(G),$$
a contradiction.
\end{proof}

\begin{lemma} \label{L042} Let $u$ be a pendant vertex of a complex unit gain graph $(G, \varphi)$ and $v$ be the vertex which is adjacent to $u$. Let $(G', \varphi)=(G, \varphi)-\{ u, v \}$. Then, $(G, \varphi)$ is upper-optimal if and only if $v$ is not on any complex unit gain cycle of $(G, \varphi)$ and $(G', \varphi)$ is upper-optimal.
\end{lemma}
\begin{proof}
(Sufficiency.) By the condition $v$ is not on any complex unit gain cycle of $(G, \varphi)$ and $(G', \varphi)$ is upper-optimal, one has that
$$c(G)=c(G'), r(G', \varphi)=2m(G')+c(G').$$
Then, by Lemmas \ref{L13} and \ref{L19}, we have
\begin{eqnarray*}
r(G, \varphi)&=&r(G', \varphi)+2\\
&=&2m(G')+c(G')+2\\
&=&2m(G)+c(G).
\end{eqnarray*}

(Necessity.) By Lemmas \ref{L13} and \ref{L19}, one has that
$$r(G, \varphi)=r(G', \varphi)+2, m(G')=m(G)+1.$$
By the condition $(G, \varphi)$ is upper-optimal, i.e., $r(G, \varphi)=2m(G)+c(G)$, we have
$$r(G', \varphi)=2m(G')+c(G).$$
It follows from Theorem \ref{T30} that
$$r(G', \varphi) \leq 2m(G')+c(G').$$
Obviously, $c(G') \leq c(G)$. Then we have
$$c(G)=c(G'), r(G', \varphi) = 2m(G')+c(G').$$
This completes the proof.
\end{proof}

\begin{lemma} \label{L043}
Let $(G, \varphi)$ be a complex unit gain unicyclic graph which contains the unique complex unit gain cycle $(C_{l}, \varphi)$.
Then $(G, \varphi)$ is upper-optimal if and only if $Re(\varphi(C_{l}, \varphi)) \neq 0$ and $l$ is odd, and $m(T_{G})=m(G-O(G))$, where $O(G)$ is the set of vertices in cycles of $(G, \varphi)$.
\end{lemma}
\begin{proof}
(Sufficiency.) Let $P_{(G, \varphi)}(\lambda)=|\lambda I_{n}-H(G, \varphi)|=\lambda^{n}+a_{1} \lambda^{n-1}+ \cdots + a_{n}$
be the characteristic polynomial of $H(G, \varphi)$ and $m=m(G)$. By Theorem \ref{T30}, we just need to prove $a_{2m+1} \neq 0$.

Since $l$ is odd and $m(T_{G})=m(G-O(G))$, by Lemma \ref{L22}, we have
$$m(G)=m(C_{l})+m(G-O(G))$$
which is equivalent to
$$2m+1=l+2m(G-O(G)).$$
Let $M_{0}$ be a maximum matching of $G-O(G)$, then $|M_{0}|=m(G-O(G))$. Tt can be checked that the order of $M_{0} \cup C_{l}$ is $2m(G-O(G))+l=2m+1$. Then, $M_{0} \cup C_{l}$ is an elementary subgraph with $2m+1$ vertices. Since $2m+1$ is odd and $(G, \varphi)$ is an unicyclic graph, each elementary subgraph with $2m+1$ vertices must contain $(C_{l}, \varphi)$ as its component. By similar method with Lemma \ref{L540},
one has that
\begin{eqnarray*}
(-1)^{2m+1}a_{2m+1}&=&\sum\limits_{(U, \varphi) \in \mathscr{U}_{2m+1}} (-1)^{p(U, \varphi)} \cdot 2^{c(U, \varphi)} \cdot Re(\varphi(C_{l}, \varphi))\\
&=&\sum\limits_{(U, \varphi) \in \mathscr{U}_{2m+1}} (-1)^{\frac{2m+1-l}{2}} \cdot 2^{1} \cdot Re(\varphi(C_{l}, \varphi)) \neq 0,
\end{eqnarray*}
where $\mathscr{U}_{2m+1}$ is the set of all elementary subgraphs contains in $(G, \varphi)$ which have exactly $2m+1$ vertices. Moreover, $p(U, \varphi)$ and $c(U, \varphi)$ are the number of even cycles and the number of cycles of $(U, \varphi)$, respectively.

(Necessity.) By the condition $(G, \varphi)$ is upper-optimal, i.e., $r(G, \varphi) = 2m+1$, we have
$a_{2m+1} \neq 0$. By similar method with Lemma \ref{L540}, one has that there exists at least one elementary subgraph of order $2m+1$. Since $2m+1$ is odd and $(G, \varphi)$ is an unicyclic graph, each elementary subgraph of order $2m+1$ must contain the unique cycle $(C_{l}, \varphi)$ as its connected component.
Moreover, $l$ is odd and $Re(\varphi(C_{l}, \varphi)) \neq 0$.

Next, we show $m(T_{G})=m(G-O(G))$ by induction on $|V(T_{G})|$. Since $(G, \varphi)$ contains a cycle, $|V(T_{G})| \geq 1$.  If $|V(T_{G})|=1$, then $(G, \varphi) \cong (C_{l}, \varphi)$, the result follows trivially. Now one can suppose that $|V(T_{G})| \geq 2$. Then there exists a pendant
vertex $u$ of $T_{G}$ which is also a pendant vertex of $(G, \varphi)$. Let $v$ be the unique neighbour
of $u$ and $(G', \varphi)=(G, \varphi)-\{u, v  \}$. By Lemma \ref{L042}, $(G', \varphi)$ is also upper-optimal and $v$ is not on any complex unit gain cycle of $(G, \varphi)$. Since $|V(T_{G'})| < |V(T_{G})| $, by induction hypothesis, one has
$$m(T_{G'})=m(G'-O(G')).$$
By Lemma \ref{L19}, we have
$$m(T_{G})=m(T_{G'})+1=m(G'-O(G'))+1=m(G-O(G)).$$
The result follows.
\end{proof}

\begin{lemma} \label{L044}
Let $(G, \varphi)$ be a complex unit gain graph without pendant vertex and $c(G) \geq 2$. Then $(G, \varphi)$
is not upper-optimal.
\end{lemma}
\begin{proof} Assume that there exists a vertex $u$ of $(G, \varphi)$ such that $c(G-u) \leq c(G)-3$.
Suppose to the contrary that $(G, \varphi)$ is upper-optimal, by Lemmas \ref{L17} and \ref{L14} and Theorem \ref{T30}, one has that
\begin{eqnarray*}
c(G)&=&r(G, \varphi)-2m(G)\\
&\leq&r((G, \varphi)-u)+2-2m(G-u)\\
&\leq& c(G-u)+2\\
&\leq& c(G)-1,
\end{eqnarray*}
a contradiction.

Now suppose that one can suppose that for any vertex $u$, $c(G-u) \geq c(G)-2$.
By Lemma \ref{L24}, there are at most
$c(G)-1$ vertices of $(G, \varphi)$ which are not covered by its maximum matching. Then,
$$m(G) \geq \frac{|V(G, \varphi)|-c(G)+1}{2}.$$
Suppose to the contrary that $(G, \varphi)$ is upper-optimal. Then one has that
$r(G, \varphi)=2m(G)+c(G) \geq |V(G)|+1.$
This contradiction completes the proof of the lemma.
\end{proof}

\begin{lemma} \label{L045}
Let $a_{j}+b_{j}\boldsymbol{i}$ (here $\boldsymbol{i}$ is the imaginary number unit) be a complex number with $|a_{j}+b_{j}\boldsymbol{i}|=1$ and $a_{j} \neq 0$ for each $j \in \{ 1, 2, \cdots, k  \}$.
Then $\sum \prod_{j=1}^{k}x_{j}=2^{k}\prod_{j=1}^{k}a_{j}$, where $x _{j} \in \{a_{j}+b_{j}\boldsymbol{i}, a_{j}-b_{j}\boldsymbol{i}  \}$ and the sum of $\sum \prod_{j=1}^{k}x_{j}$ over all the $2^{k}$ different situations.
\end{lemma}
\begin{proof} We argue by induction on $k$ to show the result. If $k=1$, then the result follows immediately.
Suppose that the result holds for any integer number $s < k$.
Then, one has that $\sum \prod_{j=1}^{k-1}x_{j}=2^{k-1}\prod_{j=1}^{k-1}a_{j}$.
Moreover, it can be checked that $$\sum \prod_{j=1}^{k}x_{j}=(\sum \prod_{j=1}^{k-1}x_{j})(a_{k}+b_{k}\boldsymbol{i})+(\sum \prod_{j=1}^{k-1}x_{j})(a_{k}-b_{k}\boldsymbol{i}).$$
Thus, by direct calculation, one has
$$\sum \prod_{j=1}^{k}x_{j}=2^{k}\prod_{j=1}^{k}a_{j}.$$
\end{proof}

Now, we give the proof of the main result of this section.

\noindent
{\bf The proof of Theorem \ref{T40}.}
(Sufficiency.) Let $(G, \varphi)$ be a complex unit gain graph which satisfies all the conditions of (i)-(iii). Let
$P_{(G, \varphi)}(\lambda)=|\lambda I_{n}-H(G, \varphi)|=\lambda^{n}+a_{1} \lambda^{n-1}+ \cdots + a_{n}$
be the characteristic polynomial of $H(G, \varphi)$, $m(G)$ and $c(G)$ are simply written as $m$ and $c$, respectively.
By Theorem \ref{T30}, it suffices to show that $a_{2m+c} \neq 0$.

By Lemma \ref{L15}, we may assume that $(G, \varphi)$ contains at least one cycle.
By Lemma \ref{L22}, we have
$m(G)=\sum_{C \in \mathscr{L}(G, \varphi)}m(C)+m(G-O(G))$.
Let $(O_{1}, \varphi), (O_{2}, \varphi), \cdots, (O_{c}, \varphi)$ be all cycles of $(G, \varphi)$ and $M_{1}$ be a maximum
matching of $G-O(G)$. Then, it can be checked that $(\cup_{j=1}^{c} (O_{j}, \varphi)) \cup M_{1}$ is an elementary subgraph with order $2m+c$.
Consequently, the set of all elementary subgraphs with order $2m+c$ is not empty. Now suppose that $(U, \varphi)$ is an elementary subgraph of order
$2m+c$ with $(O_{i_{1}}, \varphi), (O_{i_{2}}, \varphi), \cdots, (O_{i_{k}}, \varphi), K_{2}^{1}, K_{2}^{2},  \cdots,  K_{2}^{q}$ as all of its connected components, where $(O_{i_{j}}, \varphi)$ ($j \in \{ 1, 2, \cdots, k \}$) denotes an odd cycle and $K_{2}^{h}$ ($h \in \{ 1, 2, \cdots, q \}$) denotes an edge. Obviously,
$$|V(O_{i_{1}})|+|V(O_{i_{2}})|+\cdots +|V(O_{i_{k}})|+2q=2m+c.$$

Note that $m \geq m(U, \varphi)$. Hence, one has that
$$m \geq \frac{|V(O_{i_{1}})|-1}{2}+\frac{|V(O_{i_{2}})|-1}{2}+\cdots+
\frac{|V(O_{i_{k}})|-1}{2}+q=\frac{2m+c-k}{2},$$
which implies that $k \geq c$, thus we have $k=c$. Therefore, each elementary subgraph of $(G, \varphi)$ with order $2m+c$ must contain all cycles of $(G, \varphi)$.

For each $j \in \{ 1, 2, \cdots, k  \}$, let $\varphi(C_{j}, \varphi)=a_{j}+b_{j}\boldsymbol{i}$ for some
direction of $(C_{j}, \varphi)$, then $\varphi(C_{j}, \varphi)=a_{j}-b_{j}\boldsymbol{i}$ for the
other direction of $(C_{j}, \varphi)$. By the condition (ii) and the definition of the complex unit gain graph, one has that $|a_{j}+b_{j}\boldsymbol{i}|=1$ and $a_{j} \neq 0$ for each $j \in \{ 1, 2, \cdots, k  \}$.
Since each elementary subgraph of $(G, \varphi)$ with order $2m+c$ must contain all cycles of $(G, \varphi)$, then
the number of edge components of each elementary subgraph of $(G, \varphi)$ with order $2m+c$ is a fixed number $\frac{2m+c-\sum_{i=1}^{k}|V(O_{i})|}{2}$.
By a similar discussion as Lemma \ref{L540}, one has that
$$(-1)^{2m+c}a_{2m+c}=(-1)^{\frac{2m+c-\sum_{i=1}^{k}|V(O_{i})|}{2}}\sum (\prod_{j=1}^{k}x_{j}),$$
where $x _{j} \in \{a_{j}+b_{j}\boldsymbol{i}, a_{j}-b_{j}\boldsymbol{i}  \}$ and the sum takes over all the different situations. By Lemma \ref{L045}, we have
$$(-1)^{2m+c}a_{2m+c}=(-1)^{\frac{2m+c-\sum_{i=1}^{k}|V(O_{i})|}{2}}=2^{k}(\prod_{j=1}^{k}a_{j}) \neq 0.$$

(Necessity.) We proceed by induction on the order $n$ of $(G, \varphi)$ to prove (i)-(iii). If $n=1$,
then (i)-(iii) hold trivially. Suppose that (i)-(iii) hold for all upper-optimal connected
complex unit gain graph of order smaller than $n$. Now, let $(G, \varphi)$ be an upper-optimal connected complex unit gain
graph of order $n \geq 2$.
If $c(G)=0$, then $(G, \varphi)$ is a complex unit gain tree and (i)-(iii) hold trivially. If $c(G)=1$, then $(G, \varphi)$ is a complex unit gain unicyclic graph and (i)-(iii) follow immediately from Lemma \ref{L043}. If $c(G) \geq 2$, then
by Lemma \ref{L044}, $(G, \varphi)$ has at least one pendant vertex. Let $u$ be a pendant vertex of $(G, \varphi)$ and $v$ be the unique neighbour of $u$. Denote $(G_{0}, \varphi)= (G, \varphi)- \{ u,  v \}$, then it follows from Lemma \ref{L042} that $v$ does not lie on any cycle of $(G, \varphi)$ and $(G_{0}, \varphi)$ is also upper-optimal. In view of Lemma \ref{L041}, we know that every connected component of $(G_{0}, \varphi)$ is upper-optimal. Applying induction hypothesis to every connected component of $(G_{0}, \varphi)$ yields each of the following:
\\
(c) the cycles (if any) of $(G_{0}, \varphi)$ are pairwise vertex-disjoint;
\\
(d) for each cycle (if any) $(C_{l}, \varphi)$ of $(G_{0}, \varphi)$, $Re((-1)^{\frac{l-1}{2}}\varphi(C_{l}, \varphi)) \neq 0$ and $l$ is odd;
\\
(e) $m(T_{G_{0}})=m(G-O(G_{0}))$, where $O(G_{0})$ is the set of vertices in cycles of $G_{0}$.

Note that all cycles of $(G, \varphi)$ belong to $(G_{0}, \varphi)$, then (i) and (ii) hold from (c) and (d) directly.
Moreover, it can be checked that $u$ is also a pendant vertex of $T_{G}$ (resp., $G-O(G)$) which is adjacent to $v$ and $T_{G_{0}}=T_{G}-\{ u, v \}$ (resp., $G_{0}-O(G_{0})=G-O(G)-\{ u, v \}$).
Thus, by Lemma \ref{L19} and assertion (e), one has $$m(T_{G})= m(T_{G_{0}})+1= m(G_{0}-O(G_{0}))+1=m(G-O(G)).$$
This completes the proof of Theorem \ref{T40}.

\section{Proof of Theorem \ref{T60}.}


To prove Theorem \ref{T60}, we first establish the following
result.

\begin{lemma}\label{SL1}
Let $G=(V, E)$ be any simple graph with $V(G)=\{v_1, v_2, \cdots, v_n  \}$ and $A$ be any $n \times n$ matrix $(a_{i, j})$ such that $a_{i, j}=0$ whenever $v_{i}v_{j} \notin E(G)$.
For any proper subset $S$ of $V$, if $G-S$ is bipartite,
then $r(A) \leq 2m(G)+|S|$ holds.
\end{lemma}

\begin{proof} Assume that $s=r(A)$.
Then $A$ contains an $s\times s$ sub-matrix $A_0$
such that $\det (A_0)\ne 0$.
Let $R=\{i_1,i_2,\cdots,i_s\}$ and  $C=\{j_1,j_2,\cdots,j_s\}$
be the set of row numbers and the set of column numbers of $A_0$
respectively.
Then,
\begin{equation} \label{ES1}
\det (A_0)
=\sum\limits_{\pi}{\rm{sgn}}(\pi)
\prod_{t=1}^{s}a_{i_t, j_{\pi(t)}},
\end{equation}
where the sum runs over all permutations
of $1, 2, \cdots, s$ and  ${\rm{sgn}}(\pi)$ is a number in $\{-1,1\}$.
As $\det (A_0)\ne 0$,
there exists a permutation $\pi$
of $1, 2, \cdots, s$ such that
$\prod_{t=1}^{s}a_{i_t, j_{\pi(t)}}
\neq 0$, i.e., $a_{i_t, j_{\pi(t)}} \neq 0$
for all $t= 1, 2, \cdots, s$.
By the definition of $A$,
$v_{i_t}$ is adjacent to $v_{j_{\pi(t)}}$ for all $t=1,2,\cdots,s$.

Let $E_0=\{v_{i_t}v_{j_{\pi(t)}}: t=1,2,\cdots,s\}\subseteq E$,
$G_0$ be the spanning subgraph with edge set $E_0$
and $W=\{(i_t,j_{\pi(t)}): t=1,2,\cdots,s\}$.

\noindent {\bf Claim 1}:
$G_0$ has exactly $s$ edges.

It follows from the definition of $G_0$.

\noindent {\bf Claim 2}:
For any $a\in \{1,2,\cdots,n\}$,
$|\{(a,j_p)\in W: 1\le p\le s\}|\le 1$
and  $|\{(i_p,a)\in W: 1\le p\le s\}|\le 1$.

If $(a,j_p)\in W$, then $a=i_t$ for some $t:1\le t\le s$ and $p=\pi(t)$.
As $i_1,i_2,\cdots,i_s$ are  pairwise distinct,
$t$ is unique.
As $\pi$ is a permutation of $1,2,\cdots,s$,  $p=\pi(t)$
is also unique, implying that
$|\{(a,j_p)\in W: 1\le p\le s\}|\le 1$ holds.
Similarly, if $(i_p,a)\in W$, then $a=j_{\pi(p)}$.
As $\pi$ is a permutation of $1,2,\cdots,s$
and $j_1,j_2,\cdots,j_s$ are  pairwise distinct,
$p$ is the unique number in $\{1,2,\cdots,s\}$ such that
$a=j_{\pi(p)}$.
Thus $|\{(i_p,a)\in W: 1\le t\le s\}|\le 1$ holds.

\noindent {\bf Claim 3}:
Each non-trivial component of $G_0$ is either a cycle or
a path of length at least $2$.

Note that
a component of $G_0$ is said to be trivial if it is an isolated
vertex of $G_0$.
To prove this claim,
it suffices to show that $\Delta(G_0)\le 2$ holds.
Suppose that $\Delta(G_0)\ge 3$.
Without loss of generality, assume that $d_{G_0}(v_1)\ge 3$ and
$v_1v_b\in E(G_0)$ for $b=2,3,4$.
Then $|\{(1,b),(b,1)\}\cap W|\ge 1$ for all $b=2,3,4$,
implying that
$$
|\{(1,j_p)\in W: 1\le p\le s\}|
+|\{(i_p,1)\in W: 1\le p\le s\}|\ge 3,
$$
which contradicts Claim 2.

\noindent {\bf Claim 4}:
For any $S\subset V$, if $G-S$ is bipartite, then
$s\le 2m(G)+|S|$.

As $s=|E_0|$, it suffices to show that $|E_0|\le 2m(G)+|S|$.
By Claim 3, each non-trivial component of $G_0$ is either a cycle
or a path.

If a cycle $C$ is a component of $G_0$, then either
$|E(C)|=2m(C)+1$ or $|E(C)|=2m(C)$,
where $|E(C)|=2m(C)+1$ if and only if $C$ is an odd cycle.

If a path $P$ is a component of $G_0$, then either
$|E(P)|=2m(P)-1$ or $|E(P)|=2m(P)$,
where $|E(P)|=2m(P)-1$ if and only if $|E(P)|$ is odd.
Thus,
\begin{equation} \label{ES2}
|E_0|\le 2m(G_0)+oc(G_0),
\end{equation}
where $oc(G_0)$ is the number of components in $G_0$
which are odd cycles.

Assume that $\alpha=oc(G_0)$ and $C_{1}, C_{2}, \cdots, C_{\alpha}$ are the
components of $G_0$ which are odd cycles.
Since $G-S$ is bipartite,
$S \cap V(C_{i}) \ne \emptyset$ holds for all $ i= 1, 2, \cdots, \alpha$, implying that $|S| \geq \alpha$.
Thus, by (\ref{ES2}),
\begin{equation} \label{ES3}
s\le 2m(G_0)+oc(G_0) \leq 2m(G)+\alpha \leq 2m(G)+|S|.
\end{equation}
Hence Claim 4 holds, and the result follows immediately.
\end{proof}

Let $b(G)$ be the minimum integer $|S|$ such that $G-S$ is bipartite for $S \subset V(G)$.
We are now ready to prove Theorem~\ref{T60}.

\vspace{0.3 cm}

{\bf The proof of Theorem \ref{T60}:}
Let $V_0$ be any proper subset of $V$ such that
$G-V_0$ is acyclic.
By Lemmas~\ref{L14} and~\ref{L15},
$r(G,\varphi)\ge r(G-V_0,\varphi)=2m(G-V_0)$.
Thus, the lower bound of $r(G,\varphi)$ in Theorem~\ref{T60} holds.
The upper bound of $r(G,\varphi)$ in Theorem~\ref{T60}
follows directly from Lemma~\ref{SL1}.
This completes the proof of Theorem \ref{T60}.

\vskip0.2cm

$\mathbf{Remark.}$ In this paper, a lower bound and an upper bound for
$r(G,\varphi)$ in terms of $c(G)$ and $m(G)$ are obtained respectively.
Moreover, the properties of the extremal graphs which attain the lower and upper bounds are investigated.
Theorems~\ref{T30}, \ref{T50}, \ref{T40} generalize the corresponding results about undirected graphs, mixed graphs and signed graphs, which were obtained in \cite{WANGLONG}, \cite{LSC} and \cite{HSJ}, respectively.

It is not difficult to prove that $b(G)\le c(G)$ holds and
there exists $V_0\subseteq V(G)$ such that $m(G)-c(G)\le m(G-V_0)$.
Thus the lower bound and the upper bound of Theorem \ref{T60} are
better than the corresponding bounds of Theorem \ref{T30}.
The following is an example which compares the two results.

For the graph $G$ in Figure 1, $m(G)=3$ and $c(G)=2$.
By Theorem \ref{T30}, we have $r(G,\varphi)\le 2\times 3+2=8$
and $r(G,\varphi)\ge 2\times 3-2\times 2=2$.
As $G$ is bipartite, $b(G)=0$.
For $V_0=\{u,v\}$, $m(G-V_0)=3$.
Thus, $2m(G-V_0)=6=2m(G)-b(G)$, implying that $r(G,\varphi)=6$ by
Theorem~\ref{T60}.

\begin{center}   \setlength{\unitlength}{0.7mm}
\begin{picture}(45,60)

\put(25,10){\circle*{2}}
\put(25,30){\circle*{2}}
\put(-15,30){\circle*{2}}
\put(65,30){\circle*{2}}
\put(5,50){\circle*{2}}
\put(5,10){\circle*{2}}
\put(45,50){\circle*{2}}
\put(45,10){\circle*{2}}

\put(25,30){\line(0,-1){20}}
\put(25,30){\line(1,1){20}}
\put(25,30){\line(1,-1){20}}
\put(25,30){\line(-1,1){20}}
\put(25,30){\line(-1,-1){20}}

\put(-15,30){\line(1,-1){20}}
\put(-15,30){\line(1,1){20}}
\put(65,30){\line(-1,1){20}}
\put(65,30){\line(-1,-1){20}}

\put(4,52){$u$}
\put(44,52){$v$}

\put(-3,-5){Figure 1. The graph $G$}

\end{picture} \end{center}
\vskip0.3cm

Although the bounds for $r(G,\varphi)$
in Theorem \ref{T60} are better than the
corresponding bounds in  Theorem \ref{T30},
it is difficult to characterize all the extremal graphs that achieve the lower bound and upper bound in Theorem \ref{T60}.
For example, the extremal graphs that achieve the lower bound
and upper bound of Theorem \ref{T30} also satisfy the corresponding lower and upper bounds of Theorem \ref{T60}; some bipartite graphs such as trees achieve the lower bound of Theorem \ref{T60}, but some bipartite graphs such as 4-cycle and 8-cycles do not. Furthermore, the graph obtained by identifying a vertex of two 3-cycle (or a 3-cycle and a 7-cycle) achieves the upper bound of Theorem \ref{T60}, but
the graph obtained by identifying a vertex of a 3-cycle and a 5-cycle does not. It would be meaningful to characterize all the extremal graphs that achieve the lower bound and upper bound of Theorem \ref{T60}.

\section*{Acknowledgments}

This work was supported by the National Natural Science Foundation of China
(Nos. 11971054, 11731002) and the 111 Project of China (B16002).

\end{document}